\documentclass{amsart}
\usepackage{graphicx}
\usepackage{amssymb}

\newtheorem*{Beurling}{Beurling's Criterion}
\newtheorem*{Fuglede}{Fuglede's Lemma}
\newtheorem{theorem}{Theorem}
\newtheorem{corollary}{Corollary}
\newtheorem{lemma}{Lemma}
\theoremstyle{remark}
\newtheorem{remark}{Remark}
\newtheorem{example}{Example}

\newcommand{\RR}{\mathbb{R}}
\newcommand{\RS}{\hat{\mathbb{C}}}
\newcommand{\Mod}{\mathop\mathrm{mod}\nolimits}
\newcommand{\Image}{\mathop\mathrm{Im}\nolimits}

\newcommand{\res}{\hbox{ {\vrule height .22cm}{\leaders\hrule\hskip.2cm} } }

\begin{document}

\title[Beurling's criterion and extremal metrics]{Beurling's criterion and extremal metrics for Fuglede modulus}
\author{Matthew Badger}
\date{November 26, 2012}
\address{Department of Mathematics\\ Stony Brook University\\ Stony Brook, NY 11794-3651}
\email{badger@math.sunysb.edu}
\thanks{The author was partially supported by an NSF postdoctoral fellowship, grant DMS-1203497.}
\begin{abstract} For each $1\leq p<\infty$, we formulate a necessary and sufficient condition for an admissible metric to be extremal for the Fuglede $p$-modulus of a system of measures. When $p=2$, this characterization generalizes Beurling's criterion, a sufficient condition for an admissible metric to be extremal for the extremal length of a planar curve family. In addition, we prove that every Borel function $\varphi:\RR^n\rightarrow[0,\infty]$ satisfying $0<\int\varphi^p<\infty$ is extremal for the $p$-modulus of some curve family in $\RR^n$.
\end{abstract}
\maketitle

\section{Introduction}

In this note we take a close look at extremal metrics for systems of measures and families of curves. Let us start by recalling Fuglede's definition of modulus \cite{F}. Fix once and for all a measure space $(X,\mathcal{M},m)$. A collection of measures $\mathbf{E}$ is a \emph{measure system} \emph{(over $\mathcal{M}$)} if each measure $\mu\in\mathbf{E}$ is defined on the $\sigma$-algebra  $\mathcal{M}$. A Borel function  $\varphi:X\rightarrow[0,\infty]$ is called a \emph{metric} and is said to be \emph{admissible} for $\mathbf{E}$ if $\int \varphi\, d\mu\geq 1$ for all $\mu\in\mathbf{E}$. (We do not identify two metrics which agree $m$-a.e.)  For each $0<p<\infty$, the  \emph{$p$-modulus} of $\mathbf{E}$ is given by $$ \Mod_p\mathbf{E} = \inf\left\{\int\varphi^p\,dm:\varphi\text{ is admissible for }\mathbf{E}\right\}$$ where $\Mod_p\mathbf{E}=\infty$ if admissible metrics for $\mathbf{E}$ do not exist.

\begin{example}\label{curves} To pick a concrete setting, take $(X,\mathcal{M},m)=(\RR^n,\mathcal{B}_n,m_n)$ where $m_n$ is the Lebesgue measure on the Borel subsets $\mathcal{B}_n$ of $\RR^n$. A \emph{(locally rectifiable) curve} $\gamma$ in $\RR^n$ is a concatenation (disjoint union) of countably many images of one-to-one Lipschitz maps $\gamma_i:[a_i,b_i]\rightarrow\RR^n$. Each image $\gamma_i([a_i,b_i])$ is called a \emph{piece} of $\gamma$; curves may have disjoint or overlapping pieces. (For an alternative definition of a curve, see \cite{V}.) The \emph{trace} of a curve $\gamma$ is the set $\bigcup_i\gamma_i([a_i,b_i])$, i.e.~ the union of the pieces of $\gamma$.  For every curve $\gamma$ in $\RR^n$ there is a Borel measure $\tilde\gamma$ on $\RR^n$ such that the line integral $$\int_\gamma f\,ds=\sum_i \int_{a_i}^{b_i} f(\gamma_i(t))|\gamma'_i(t)|dt$$ is given by integration against $\tilde\gamma$, i.e.\ $\int_\gamma f\,ds = \int_{\RR^n} f\,d\tilde\gamma$ for every Borel function $f$. (By the area formula $\tilde \gamma = \sum_i \tilde\gamma_i$ where $\tilde\gamma_i=H^1\res\gamma_i([a_i,b_i])$ is the 1-dimensional Hausdorff measure restricted to the set $\gamma_i([a_i,b_i])$, e.g.~ see \cite{EG}.) For all $1\leq p<\infty$, the \emph{$p$-modulus} of a family of curves $\Gamma$ in $\RR^n$ is defined in terms of Fuglede modulus by $\Mod_p\Gamma =\Mod_p\left\{\tilde\gamma: \gamma\in \Gamma\right\}$. Therefore, $\Mod_p\Gamma=\inf_\varphi \int_{\RR^n}\varphi^p$ $dm_n$ where the infimum runs over all Borel functions $\varphi\geq 0$ such that $\int_\gamma\varphi\,ds\geq 1$ for all $\gamma\in\Gamma$. (One could similarly define $\Mod_p\Gamma$ for $0<p<1$, but this quantity is always zero.) In the plane, the \emph{extremal length} $\lambda(\Gamma)=1/\Mod_2\Gamma$ of a curve family $\Gamma$ in $\RR^2$ is often used instead of its modulus. \end{example}

An \emph{atom} in a $\sigma$-algebra $\mathcal{M}$ is a nonempty set $A\in\mathcal{M}$ with the property $B\subsetneq A$, $B\in\mathcal{M}\Rightarrow B=\emptyset$. That is, the only proper measurable subset of an atom is the empty set. If $\varphi:X\rightarrow[0,\infty]$ is a Borel function on $(X,\mathcal{M})$, then $\varphi$ is constant on each atom of $\mathcal{M}$. Given an atom $A\in\mathcal{M}$, the \emph{atomic measure} $\delta_A$ is defined by the rule $\delta_A(S)=1$ if $A\subset S$ and $\delta_A(S)=0$ otherwise;  $\int \varphi\,d\delta_A=\varphi(A)$ for all $\varphi$ and $A$.

\begin{example}\label{transboundary} Let $\mathcal{K}=\{K_1,\dots, K_\ell\}$ be a finite set of pairwise disjoint compact subsets of the Riemann sphere $\RS$, and let $\Omega\subset\RS$ be an open set. The \emph{transboundary measure space} $(\RS,\mathcal{M}_\mathcal{K},m_{\Omega,\mathcal{K}})$ is defined as follows. Let $\mathcal{B}(\RS\setminus K)$ denote the Borel $\sigma$-algebra on the complement of $K=\bigcup_{i=1}^{\ell}K_i$. Then $\mathcal{M}_\mathcal{K}$ is the smallest $\sigma$-algebra generated by $\mathcal{B}(\RS\setminus K)\cup\mathcal{K}$. The atoms of $\mathcal{M}_{\mathcal{K}}$ are the singletons $\{x\}$ with $x\in\RS\setminus K$ and the sets $K_1,\dots, K_\ell$. We define the measure $m_{\Omega,\mathcal{K}}=H^2\res (\Omega\setminus K) + \sum_{i=1}^{\ell}\delta_{K_i}$ where $H^2\res(\Omega\setminus K)$ is 2-dimensional Hausdorff measure on $\Omega\setminus K$.
Let $\gamma:[a,b]\rightarrow\RS$ be a one-to-one continuous map, let $\Image\gamma=\gamma([a,b])$ be its image, and assume that $\Image\gamma\cap (\Omega\setminus K)$ is locally rectifiable. Then we define a measure $\hat \gamma$ on $(\RS,\mathcal{M}_{\mathcal{K}})$ by $$\hat{\gamma}=H^1\res\Image\gamma\cap(\Omega\setminus K) + \sum_{i:\,\Image\gamma\cap K_i\neq\emptyset}\delta_{K_i}.$$ Suppose that $(X,\mathcal{M},m)=(\RS,\mathcal{M}_{\mathcal{K}},m_{\Omega,\mathcal{K}})$. The \emph{transboundary modulus} $\Mod_{\Omega,\mathcal{K}}\Gamma$ of a collection $\Gamma$ of one-to-one continuous maps $\gamma:[a,b]\rightarrow\RS$ is defined via Fuglede modulus by $\Mod_{\Omega,\mathcal{K}}\Gamma=\Mod_2\{\hat \gamma:\gamma\in\Gamma$ and $\Image\gamma\cap(\Omega\setminus K)$ is locally rectifiable$\}.$ Thus, an admissible metric $\varphi:\RS\rightarrow[0,\infty]$ satisfies  $$\int_{\Image\gamma\cap(\Omega\setminus K)}\varphi\,ds + \sum_{i:\,\Image\gamma\cap K_i\neq\emptyset} \varphi(K_i)\geq 1$$ for every $\gamma\in\Gamma$ such that $\Image\gamma\cap (\Omega\setminus K)$ is locally rectifiable, and $$\Mod_{\Omega,\mathcal{K}}\Gamma=\inf_{\varphi}\int_{\Omega\setminus K} \varphi^2\,dH^2 + \sum_{i=1}^\ell \varphi(K_i)^2$$ where the infimum runs over all admissible metrics $\varphi:\RS\rightarrow[0,\infty]$. The reciprocal $\lambda_{\Omega,\mathcal{K}}(\Gamma)=1/\Mod_{\Omega,\mathcal{K}}\Gamma$ of transboundary modulus is \emph{transboundary extremal length}.
\end{example}

The definition of extremal length is due to Beurling and has roots in the classical length-area principle for conformal maps; see Jenkins \cite{J} for a historical overview. Since the introduction of extremal length by Ahlfors and Beurling \cite{AB}, the modulus of a curve family has become a widely-used tool, employed in geometric function theory \cite{A,GM,R}, quasiconformal and quasiregular mappings \cite{V,Vu}, dynamical systems \cite{KL,M}, and analysis on metric spaces \cite{Hebook,He}. The transboundary extremal length of a curve family was introduced by Schramm \cite{S} to study uniformization on countably-connected domains. Recently Bonk \cite{B} used transboundary modulus in a crucial way to obtain uniformization results on Sierpi\'nski carpets in the plane. For applications of modulus of measures, see Fuglede's original applications in \cite{F}, Hakobyan's work on the conformal dimension of sets \cite{H}, and Bishop and Hakobyan's recent paper on the frequency of dimension distortion by quasisymmetric maps \cite{BH}.

A few nice properties of modulus are apparent from the definition. First if $\mathbf{E}\subset\mathbf{F}$ then $\Mod_p\mathbf{E}\leq \Mod_p\mathbf{F}$. Second $\Mod_p\bigcup_{i=1}^\infty\mathbf{E}_i\leq \sum_{i=1}^\infty\Mod_p\mathbf{E}_i$ for any sequence of measure systems. Since $\Mod_p\emptyset=0$, this says that modulus is an outer measure on measure systems. A third useful property is that every admissible metric gives an upper bound on modulus, i.e.\ $\Mod_p\mathbf{E}\leq\int\varphi^p\,dm$ for all admissible metrics $\varphi$.

If the infimum in the definition of the modulus of a measure system $\mathbf{E}$ is obtained by an admissible metric $\varphi$, i.e.\ if $\Mod_p\mathbf{E}=\int\varphi^p\,dm$, then the metric $\varphi$ is said to be \emph{extremal} for the $p$-modulus of $\mathbf{E}$. Naturally one may ask whether an extremal metric always exists, and if so, to what extent is an extremal metric uniquely determined. Unfortunately simple examples (see Example \ref{zero} below) show that the existence and uniqueness of extremal metrics fails for general measure systems. Nevertheless, Fuglede \cite{F} proved that when $1<p<\infty$, a measure system always admits an extremal metric, after removing an exceptional system of measures.

\begin{Fuglede} Let $1<p<\infty$. Let $\mathbf{E}$ be a measure system. If $\Mod_p\mathbf{E}<\infty$, then there exists a measure system $\mathbf{N}\subset\mathbf{E}$ such that $\Mod_p\mathbf{N}=0$ and $\mathbf{E}\setminus\mathbf{N}$ admits an extremal metric $\varphi$.\end{Fuglede}

The uniqueness of an extremal metric for the $p$-modulus of a measure system also holds when $1<p<\infty$, up to redefinition of the metric on a set of $m$-measure zero. This can be seen as follows. Suppose that $\varphi,\psi\in L^p(m)$ are two extremal metrics for the $p$-modulus of a measure system $\mathbf{E}$. Then the averaged metric $\chi=\frac12\varphi+\frac12\psi$ is still admissible for $\mathbf{E}$ and $(\Mod_p\mathbf{E})^{1/p}\leq \|\chi\|_p \leq \frac{1}{2}\|\varphi\|_p+\frac{1}{2}\|\psi\|_p = (\Mod_p\mathbf{E})^{1/p}$. Thus, $\|\frac{1}{2}\varphi+\frac{1}{2}\psi\|_p=\frac{1}{2}\|\varphi\|_p+\frac{1}{2}\|\psi\|_p$. By the condition for equality in Minkowski's inequality and the assumption that $\|\varphi\|_p=\|\psi\|_p<\infty$, one obtains $\varphi=\psi$ $m$-a.e., as desired.

A fundamental problem working with modulus is to identify an extremal metric for a given measure system or curve family if one exists. Beurling found a general sufficient condition which guarantees that an admissible metric for a curve family in the plane is extremal for its extremal length.

\begin{Beurling}[Ahlfors \cite{A}, Theorem 4.4] Let $\Gamma$ be a curve family in $\RR^2$ and let $\varphi$ be an admissible metric for $\Gamma$ such that $0<\int_{\RR^2} \varphi^2<\infty$. Suppose that there exists a curve family $\Gamma_0$ in $\RR^2$ such that \begin{enumerate}
    \item $\Gamma_0\subset\Gamma$,
    \item $\int_\gamma \varphi\, ds=1$ for every $\gamma\in\Gamma_0$, and
    \item for all $f\in L^2(\RR^2)$ taking values in $[-\infty,\infty]$: if $\int_\gamma f\,ds\geq 0$ for all $\gamma\in\Gamma_0$, then $\int_{\RR^2} f\varphi\geq 0$.
\end{enumerate}
Then $\varphi$ is an extremal metric for the extremal length of $\Gamma$, i.e.\ $\lambda(\Gamma)=\left(\int_{\RR^2} \varphi^2\right)^{-1}$.
\end{Beurling}

Let us see Beurling's criterion in action, in a standard example.

\begin{example}\label{rect1} Let $R$ be a rectangle with side lengths $a\leq b$. Let $\Gamma$ be the family of all curves in $R$ with connected trace which join opposite edges in $R$ (see Figure 1). We claim that $\varphi=\frac{1}{a}\chi_R$ is an extremal metric for $\Gamma$, and thus, $$\lambda(\Gamma)=\left(\int_R \frac{1}{a^2}\right)^{-1}=a/b.$$ First $\varphi$ is admissible for $\Gamma$, because every curve connecting opposite edges in $R$ travels at least Euclidean distance $a$ (the distance between the edges of length $b$). Beurling's criterion holds with $\Gamma_0=\{\gamma(t):t\in[0,b]\}$ equal to the family of straight line segments connecting (and orthogonal to) a pair of opposite sides of length $b$. Conditions (1) and (2) hold by definition. And (3) follows from Fubini's theorem: if $\int_{\gamma(t)} f\,ds\geq 0$ for all $\gamma(t)\in\Gamma_0$, then $\int_{\RR^2} f\varphi = \frac{1}{a}\int_R f=\frac{1}{a}\int_0^b \int_{\gamma(t)}f\,ds\,dt\geq 0.$ Therefore, $\varphi$ is extremal for $\lambda(\Gamma)$. \begin{figure}\includegraphics[width=.7\textwidth]{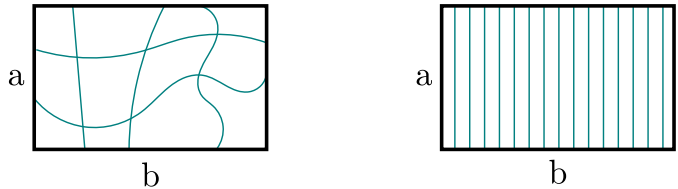}\caption{Curve families $\Gamma$ and $\Gamma_0$ in Examples 3 and 4}\end{figure}\end{example}

The converse to Beurling's criterion fails for the simple reason that $\Gamma$ may not contain any curves $\gamma$ such that $\int_{\gamma} \varphi\,ds=1$.

\begin{example} Once again let $R$ be a rectangle with side lengths $a\leq b$, and let $\Gamma$ and $\Gamma_0$ be the curve families from Example \ref{rect1}. We claim that $\varphi=\frac{1}{a}\chi_R$ is an extremal metric for $\Gamma_*=\Gamma\setminus\Gamma_0$. However, since $\Gamma_*$ does not contain any curves $\gamma$ such that $\int_{\gamma}\varphi\,ds=1$, we cannot use Beurling's criterion to show that $\varphi$ is extremal for $\Gamma_*$. Let $\psi$ be an admissible metric for $\Gamma_*$. Fix $\gamma(t)\in\Gamma_0$. Then one can find a sequence of curves $\gamma^k(t)\in\Gamma_*$ such that $\int_{\gamma^k(t)}\psi\,ds\rightarrow\int_{\gamma(t)}\psi\,ds$. (For example, if $\gamma(t)=[0,a]$, then take $\gamma^k(t)=[1/k,0]\sqcup[0,a]$ where $\sqcup$ denotes concatenation.) In particular, it follows that $\int_{\gamma(t)}\psi\,ds \geq 1$ for all $\gamma(t)\in \Gamma_0$. Integrating across all $\gamma(t)\in\Gamma_0$, invoking Fubini's theorem and applying the Cauchy-Schwarz inequality gives $$b \leq \int_0^b\int_{\gamma(t)}\psi\,ds\,dt = \int_R\psi\leq\left(\int_R\psi^2\right)^{1/2}(ab)^{1/2}\leq \left(\int_{\RR^2}\psi^2\right)^{1/2}(ab)^{1/2}.$$ Thus, $\left(\int_{\RR^2} \psi^2\right)^{-1}\leq a/b$ for every metric $\psi$ that is admissible for $\Gamma_*$. Since this upper bound is obtained by $\varphi$, we conclude that $\varphi$ is extremal for $\lambda(\Gamma_*)$.\end{example}

A partial converse to Beurling's criterion is presented in Ohtsuka \cite[\S2.3]{O} in the special case $\Gamma=\Gamma_0$: \emph{if $\varphi$ is extremal for $\Gamma$, then (3) holds for all $f\in L^2(\RR^2)$}. Wolf and Zwiebach \cite[p.~38]{WZ} have also established ``a partial local converse to Beurling's criterion" for certain classes of metrics on Riemann surfaces.

\section{Statement of Results}

The main goal of this note is to show that Beurling's criterion can be modified to become a necessary and sufficient test for extremal metrics. In fact, we establish a characterization of extremal metrics in the general setting of Fuglede modulus, when $1<p<\infty$ and when $p=1$.

\begin{theorem}[Extremal Metrics in $L^p$] \label{LpBeurling} Let $1<p<\infty$. Let $\mathbf{E}$ be a measure system and let $\varphi$ be an admissible metric for $\mathbf{E}$ such that $\varphi\in L^p(m)$. Then $\varphi$ is extremal for the $p$-modulus of $\mathbf{E}$ if and only if
\begin{enumerate}
\item[$(B_p)$] There exists a measure system $\mathbf{F}$ such that
   \begin{enumerate}
       \item $\Mod_p\mathbf{E}\cup\mathbf{F}=\Mod_p\mathbf{E}$,
       \item $\int\varphi\, d\nu=1$ for every $\nu\in\mathbf{F}$, and
       \item for all $f\in L^p(m)$ taking values in $[-\infty,\infty]$: if $\int\! f\,d\nu\geq 0$ for all $\nu\in\mathbf{F}$, then $\int f\varphi^{p-1}dm\geq 0$.
   \end{enumerate}
\end{enumerate}
\end{theorem}

\begin{theorem}[Extremal Metrics in $L^1$]\label{L1Beurling} Let $\mathbf{E}$ be a measure system and let $\varphi$ be an admissible metric for $\mathbf{E}$ such that $\varphi\in L^1(m)$. Then $\varphi$ is extremal for the $1$-modulus of $\mathbf{E}$ if and only if
\begin{enumerate}
\item[$(B_1)$] There exists a measure system $\mathbf{F}$ such that
   \begin{enumerate}
       \item $\Mod_1\mathbf{E}\cup\mathbf{F}=\Mod_1\mathbf{E}$,
       \item $\int\varphi\, d\nu=1$ for every $\nu\in\mathbf{F}$, and
       \item for all $f\in L^1(m)$ taking values in $[-\infty,\infty]$ such that $\varphi(x)=0$ implies $f(x)\geq 0$: if $\int f\,d\nu\geq 0$ for all $\nu\in\mathbf{F}$, then $\int f\, dm\geq 0$.   \end{enumerate}
\end{enumerate}
\end{theorem}

\begin{remark} We label the conditions $(B_p)$ in Theorems \ref{LpBeurling} and \ref{L1Beurling} in honor of Beurling. When $p=2$ and $\mathbf{F}\subset\mathbf{E}$, (a) holds vacuously and $(B_2)$ is Beurling's criterion.\end{remark}

\begin{remark} In Theorems \ref{LpBeurling} and \ref{L1Beurling}, if $\varphi$ is extremal for $\Mod_p\mathbf{E}$, then there exists $\mathbf{F}$ satisfying $(B_p)$ such that for every $\nu\in\mathbf{F}$ there exist $\mu\in\mathbf{E}$ and $0<c\leq 1$ such that $\nu=c\mu$.\end{remark}

\begin{remark} In Theorems \ref{LpBeurling} and \ref{L1Beurling} the case $\mathbf{F}=\emptyset$ is allowed. Note condition $(B_p)$ holds with $\mathbf{F}=\emptyset$ if and only if $\varphi=0$ $m$-almost everywhere.\end{remark}

The proofs of Theorems \ref{LpBeurling} and \ref{L1Beurling} will be given in sections \ref{lpsect} and \ref{l1sect}, respectively. (A curious reader may jump to the proofs immediately.) We now demonstrate the use of the theorems in a simple, yet enlightening example, which shows the varied behavior of the $p$-modulus for different values of $p$.

\begin{example}\label{zero} Assume that $A\in\mathcal{M}$ and $0<m(A)<\infty$. Let $\mathbf{E}_A=\{m\res A\}$ where $m\res A$ denotes the measure $m$ restricted to the set $A$. Then $$\Mod_p\mathbf{E}_A=\left\{\begin{array}{ll}
\inf\{m(B)^{1-p}:B\subset A,\  m(B)>0\},\quad &\text{if }0<p\leq 1,\\ m(A)^{1-p}, &\text{if }1\leq p<\infty.\end{array}\right.$$
This will be checked in four steps.

Let $1<p<\infty$. We will show that $\varphi_A = m(A)^{-1}\chi_A$ is extremal for $\Mod_p\mathbf{E}_A$, and hence, $\Mod_p\mathbf{E}_A =\int_A m(A)^{-p}\,dm=m(A)^{1-p}.$ Clearly $\varphi_A\in L^p(m)$ and $\varphi_A$ is admissible for $\mathbf{E}_A$. Let us check that $(B_p)$ holds with $\mathbf{F}=\mathbf{E}_A$. Conditions (a) and (b) hold immediately. For condition (c), $\int f\varphi_A^{p-1}\,dm= m(A)^{1-p}\int_A f\,dm\geq 0$ whenever $\int f\,d(m\res A)\geq 0$. Thus, $\varphi_A$ is extremal for $\Mod_p\mathbf{E}_A$, by Theorem \ref{LpBeurling}.

The case $p=1$ is similar, except that there is no longer a unique extremal metric. Let $B\subset A$ be any subset such that $m(B)>0$ and let $\varphi_B=m(B)^{-1}\chi_B$. Then $\varphi_B\in L^1(m)$ and $\varphi_B$ is admissible for $\mathbf{E}_A$. We will check that $(B_1)$ holds with $\mathbf{F}=\mathbf{E}_A$. Conditions (a) and (b) are immediate. To verify condition (c) of $(B_1)$, assume that $f\in L^1(m)$ takes values in $[-\infty,\infty]$, $f(x)\geq 0$ whenever $\varphi_B(x)=0$ and $\int f\,d(m\res A)\geq 0$. Then $\int f\,dm = \int_{A^c} f\,dm + \int_A f\,dm\geq 0$, where the first term is non-negative since $\varphi_B(x)=0$ on $A^c$. Thus, $\varphi_B$ is extremal for $\Mod_1\mathbf{E}_A$, by Theorem \ref{L1Beurling}, so that $\Mod_1\mathbf{E}_A = \int \varphi_B\,dm = 1$ for every $B\subset A$ with $m(B)>0$.

Next let $0<p<1$ and suppose that $A$ has subsets of arbitrarily small positive measure. Then we can find a sequence of subsets $B_k\subset A$ with $m(B_k)>0$ such that $\lim_{k\rightarrow\infty} m(B_k)=0$. The normalized characteristic functions $\varphi_k=m(B_k)^{-1}\chi_{B_k}$ are  admissible for $\mathbf{E}_A$. Hence $\Mod_p\mathbf{E}_A \leq \int\varphi_k^p\,dm=m(B_k)^{1-p}\rightarrow 0$ as $k\rightarrow\infty$, since $0<p<1$. Therefore, $\Mod_p\mathbf{E}_A=0=\inf\{m(B)^{1-p}:B\subset A,\  m(B)>0\}$. However, there is no extremal metric for $\Mod_p\mathbf{E}_A$, because no function $\psi\geq 0$ satisfies $\int\psi\,d(m\res A)\geq 1$ and $\int \psi^p\,dm=0$ simultaneously.

Finally, let $0<p<1$, but suppose that $A$ does not possess subsets of arbitrarily small positive measure. Then $m\res A= c_1\delta_{A_1}+\dots+c_k\delta_{A_k}$ is a linear combination of atomic measures, where each atom $A_i\in\mathcal{M}$ and $A_i\cap A_j=\emptyset$ whenever $i\neq j$. By relabeling, we may assume that $0<c_1\leq c_j$ for all $2\leq j\leq k$. Let $\rho\geq 0$ be an admissible metric for $\mathbf{E}_A$ such that $\int\rho\,d(m\res A)=1$. (Here we can ask for equality, because $\mathbf{E}_A$ consists of one element.) Define $\eta_j=\rho(A_j)c_j$ for all $j$. Then $\sum_{j=1}^k\eta_j = \sum_{j=1}^k (\eta_j/c_j)c_j= \int\rho\,d(m\res A)=1$. Thus, $0\leq \eta_j\leq 1$ for all $j$, and  $$\int \rho^p\,dm\geq \sum_{j=1}^k(\eta_j/c_j)^pc_j=\sum_{j=1}^k \eta_j^p c_j^{1-p}\geq \sum_{j=1}^k\eta_j c_1^{1-p}=c_1^{1-p}.$$ Since the lower bound $\int\rho^p\,dm\geq c_1^{1-p}$ is obtained by the metric $\rho=m(A_1)^{-1}\chi_{A_1}$, we conclude that $\Mod_p\mathbf{E}_A=m(A_1)^{1-p}=\inf\{m(B)^{1-p}:B\subset A,\  m(B)>0\}$.
\end{example}

\begin{remark} With the same notation as in the previous example, $\varphi_A=m(A)^{-1}\chi_A$ also satisfies condition $(B_p)$ with $\mathbf{F}=\mathbf{E}_A$ for $0<p<1$. But $\Mod_p \mathbf{E}_A\neq \int \varphi_A^p\,dm$ when $0<p<1$ unless $m\res A=c\delta_{A}$. Thus, Example \ref{zero} shows that condition $(B_p)$ from Theorem \ref{LpBeurling} is not a sufficient test for extremal metrics when $0<p<1$. \end{remark}

The characterizations of extremal metrics for the $p$-modulus of measure systems in Theorems \ref{LpBeurling} and \ref{L1Beurling} also hold for curve families in $\RR^n$. In particular, assuming that an extremal metric for $\mathbf{E}=\{\tilde \gamma:\gamma\in\Gamma\}$ exists, one can find a measure system $\mathbf{F}$ satisfying condition $(B_p)$, where $\mathbf{F}$ is also associated to a family of curves in $\RR^n$.

\begin{corollary}[Extremal Metrics in $L^p$ for Curves] \label{Lpcurve} Let $1<p<\infty$. Let $\Gamma$ be a curve family in $\RR^n$ and let $\varphi$ be an admissible metric for $\Gamma$ such that $\varphi\in L^p(\RR^n)$. Then $\varphi$ is extremal for the $p$-modulus of $\Gamma$ if and only if
\begin{enumerate}
\item[$(B'_p)$] There exists a curve family $\Gamma'$ in $\RR^n$ such that
   \begin{enumerate}
       \item $\Mod_p\Gamma\cup\Gamma'=\Mod_p\Gamma$,
       \item $\int_{\gamma}\varphi\, ds=1$ for every $\gamma\in\Gamma'$, and
       \item for all $f\in L^p(\RR^n)$ taking values in $[-\infty,\infty]$: if $\int_\gamma f\,ds\geq 0$ for all $\gamma\in\Gamma'$, then $\int_{\RR^n} f\varphi^{p-1}\geq 0$.
   \end{enumerate}
\end{enumerate}
\end{corollary}

\begin{corollary}[Extremal Metrics in $L^1$ for Curves]\label{L1curve} Let $\Gamma$ be a curve family in $\RR^n$ and let $\varphi$ be an admissible metric for $\Gamma$ such that $\varphi\in L^1(\RR^n)$. Then $\varphi$ is extremal for the $1$-modulus of $\Gamma$ if and only if
\begin{enumerate}
\item[$(B'_1)$] There exists a curve family $\Gamma'$ in $\RR^n$ such that
   \begin{enumerate}
       \item $\Mod_1\Gamma\cup\Gamma'=\Mod_1\Gamma$,
       \item $\int_\gamma\varphi\, ds=1$ for every $\gamma\in\Gamma'$, and
       \item for all $f\in L^1(\RR^n)$ taking values in $[-\infty,\infty]$ such that $\varphi(x)=0$ implies $f(x)\geq 0$: if $\int_\gamma f\,ds\geq 0$ for all $\gamma\in\Gamma'$, then $\int_{\RR^n} f\geq 0$.   \end{enumerate}
\end{enumerate}
\end{corollary}

\begin{remark}\label{subcurve} In Corollaries \ref{Lpcurve} and \ref{L1curve}, if $\varphi$ is extremal for $\Mod_p\Gamma$, then there exists $\Gamma'$ satisfying $(B'_p)$ such that every curve $\gamma'\in\Gamma'$ is a \emph{subcurve} of some curve $\gamma\in\Gamma$. \end{remark}

\begin{remark} In Corollaries \ref{LpBeurling} and \ref{L1Beurling} the case $\Gamma'=\emptyset$ is allowed. Note condition $(B'_p)$ holds with $\Gamma'=\emptyset$ if and only if $\varphi=0$ Lebesgue almost everywhere.\end{remark}

The auxiliary curve family $\Gamma'$ that is required to test condition $(B'_p)$ is not unique. In the next example, we exhibit disjoint curve families $\Gamma_0$ and $\Gamma_1$ such that condition $(B_2')$ holds with the auxiliary curve family $\Gamma'=\Gamma_i$, $i=0,1$.

\begin{example}\label{unique} Let $R$ be a rectangle with side lengths $a\leq b$, and let $\Gamma$ and $\Gamma_0$ be the curve families from Example \ref{rect1}. Above we showed that condition $(B_2')$ (i.e.~ Beurling's criterion) holds for $\Gamma$ and $\varphi=\frac{1}{a}\chi_R$ using the auxiliary curve family $\Gamma'=\Gamma_0\subset\Gamma$. Thus, $\Mod_2\Gamma=\int_R(1/a)^2=b/a$ by Corollary \ref{Lpcurve}. Alternatively let $\Gamma_1$ be the curve family described as follows (see Figure 2). For each $\gamma(t)\in\Gamma_0$, there correspond exactly two curves $\gamma'(t)$ and $\gamma''(t)$ in $\Gamma_1$. If the curve $\gamma(t)=[P,Q]$, then the curves $\gamma'(t)$ and $\gamma''(t)$ are given by $$\gamma'(t)=\left[P,\frac{P+Q}{2}\right]\sqcup\left[\frac{P+Q}{2},P\right]\quad\text{ and }\quad \gamma''(t)=\left[Q,\frac{P+Q}{2}\right]\sqcup\left[\frac{P+Q}{2},Q\right]$$
where $\sqcup$ denotes concatenation. \begin{figure}\includegraphics[width=.7\textwidth]{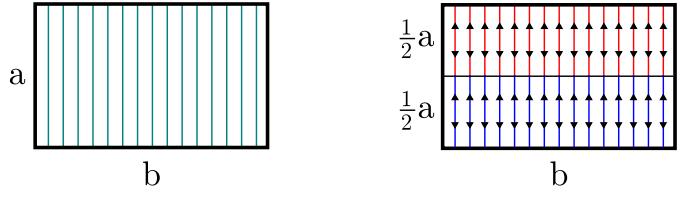}\caption{Curve families $\Gamma_0$ and $\Gamma_1$ in Example 6}\end{figure} In other words, each curve in $\Gamma_1$ travels along a straight path starting at and perpendicular to an edge of side length $b$; half-way across to the other side, the curve reverses direction and returns to its starting point. We now check that $(B_2')$ holds for $\Gamma$ and $\varphi$ with $\Gamma'=\Gamma_1$. A quick computation shows that $\int_\gamma \varphi\,ds=1$ for all $\gamma\in \Gamma_1$. Hence condition (b) holds. For (a), we have $\Mod_2\Gamma \leq \Mod_2\Gamma\cup\Gamma_1\leq \int_{\RR^2}\varphi^2=\Mod_2\Gamma$, since $\varphi$ is admissible for $\Gamma\cup\Gamma_1$ and since (we already know that) $\varphi$ is extremal for $\Gamma$. It remains to check (c). If $f\in L^2(\RR^2)$ and $\int_{\gamma} f\,ds\geq 0$ for all $\gamma\in \Gamma_1$, then \begin{equation*}\begin{split}2\int_{\RR^2} f\varphi &= \frac{2}{a}\int_R f = \frac{2}{a}\int_0^b\int_{\gamma(t)}f\,ds\,dt=\frac{1}{a}\int_0^b\int_{\gamma(t)\sqcup\gamma(t)}f\,ds\,dt\\
&=\frac{1}{a}\int_0^b\int_{\gamma'(t)\sqcup\gamma''(t)}f\,ds\,dt =\frac{1}{a}\int_0^b\int_{\gamma'(t)}f\,ds\,dt + \frac{1}{a}\int_0^b\int_{\gamma''(t)}f\,ds\,dt\geq 0.\end{split}\end{equation*} Thus, condition (c) holds too, and we have reached the end of the example.\end{example}

It is of course possible to specialize Theorems \ref{LpBeurling} and \ref{L1Beurling} to other settings. For instance, a reader familiar with analysis on metric spaces will have no difficulty adapting Corollaries \ref{Lpcurve} and \ref{L1curve} to the metric space setting. In \cite[\S11]{B}, Bonk notes that Beurling's criterion can be adapted to produce a sufficient test for extremal metrics for the transboundary modulus of a curve family in the Riemann sphere. Using Theorem \ref{LpBeurling} and the proof of Corollary \ref{Lpcurve}, one can also formulate a necessary and sufficient test for extremal metrics for transboundary modulus.

So far we have found a characterization of extremal metrics for the $p$-modulus of a measure system or curve family when $1\leq p<\infty$. A related problem is to identify those metrics which are extremal for the $p$-modulus of some measure system or curve family. The next result gives a solution to this problem for measure systems.

\begin{theorem}\label{xmetrics} If $\varphi:X\rightarrow[0,\infty]$ is a metric and $\varphi<\infty$ $m$-a.e., then $\varphi$ is extremal for the $p$-modulus of $\mathbf{E}_\varphi=\{\mu \text{ defined on }\mathcal{M}:\int\varphi\,d\mu\geq 1\}$ for all $0<p<\infty$.\end{theorem}

\begin{proof} Let $\mathcal{A}$ be the collection of atoms in $\mathcal{M}$. Note that the scaled atomic measure $\mu_A=\varphi(A)^{-1}\delta_A\in\mathbf{E}_\varphi$ for all $A\in \mathcal{A}$ such that $0<\varphi(A)<\infty$. Let $\psi$ be an admissible metric for $\mathbf{E}_\varphi$. Then $\psi(A)/\varphi(A)=\int \psi\, d\mu_A\geq 1$ when $0<\varphi(A)<\infty$. Also, $\psi(A)\geq \varphi(A)$ when $\varphi(A)=0$. Thus, if $\varphi<\infty$ $m$-a.e., then $\psi\geq\varphi$ $m$-a.e., and $\int \psi^p\,dm\geq \int \varphi^p\,dm$ for all $0<p<\infty$. Therefore, $\Mod_p\mathbf{E}_\varphi=\int \varphi^p\,dm$ for all $0<p<\infty$.\end{proof}

We can establish a similar result for curve families in $\RR^n$. The basic philosophy, suggested by the proof of Theorem \ref{xmetrics}, is that one needs to approximate the measures $\delta_x$ at points  where $\varphi(x)>0$ by sequences of curves. See section \ref{sect4} for details.

\begin{theorem}\label{curvemetrics} If $\varphi:\RR^n\rightarrow[0,\infty]$ is Borel, then $\varphi$ is extremal for the $p$-modulus of $\Gamma_{\varphi}=\{\text{curve }\gamma\text{ in }\RR^n:\int_\gamma \varphi\,ds\geq 1\}$ for all $1\leq p<\infty$ such that $0<\int_{\RR^n} \varphi^p<\infty$.\end{theorem}

The plan for the remainder of the note is as follows. In the next two sections, we prove the characterizations of extremal metrics for the $p$-modulus of a measure system from above, in the cases $1<p<\infty$ (section \ref{lpsect}) and $p=1$ (section \ref{l1sect}). Then we turn our attention to extremal metrics for families of curves in $\RR^n$. In section \ref{modify}, we show how the proofs of Theorems \ref{LpBeurling} and \ref{L1Beurling} must be modified to obtain Corollaries \ref{Lpcurve} and \ref{L1curve}. Finally, we give the proof of Theorem \ref{curvemetrics} in section \ref{sect4}.

\section{Proof of Theorem \ref{LpBeurling} (Extremal Metrics in $L^p$)}\label{lpsect}

Let $1<p<\infty$. Let $\mathbf{E}$ be a measure system and let $\varphi$ be an admissible metric for $\mathbf{E}$ such that $\varphi\in L^p(m)$. If $\varphi=0$ $m$-a.e., then $\varphi$ is extremal for the $p$-modulus of $\mathbf{E}$ and condition $(B_p)$ holds with $\mathbf{F}=\emptyset$. Thus, we assume that $0<\int \varphi^p\,dm<\infty$.

We shall start with the proof that condition $(B_p)$ implies that $\varphi$ is extremal, by mimicking the proof of Beurling's criterion in Ahlfors \cite{A}. Suppose that $(B_p)$ holds for some measure system $\mathbf{F}$ satisfying (a), (b) and (c). Since the metric $\varphi$ is admissible for $\mathbf{E}$, $\varphi$ is also admissible for $\mathbf{E}\cup\mathbf{F}$, by (b). Let $\psi$ be a competing admissible metric for $\Mod_p\mathbf{E}\cup\mathbf{F}$, so that $\int \psi^p\,dm \leq \int \varphi^p\,dm<\infty$. Then $\int \psi\,d\nu \geq 1=\int \varphi\, d\nu$ for all $\nu\in\mathbf{F}$, by (b). Hence $f=\psi-\varphi\in L^p(m)$ and $\int f\,d\nu\geq 0$ for all $\nu\in\mathbf{F}$. By (c), we conclude that $\int(\psi-\varphi)\varphi^{p-1}\geq 0$. Then \begin{equation}\label{holder}\int \varphi^p\, dm\leq \int \psi\varphi^{p-1}\, dm\leq \left(\int \psi^p\,dm\right)^{1/p}\left(\int \varphi^p\,dm\right)^{(p-1)/p}\end{equation} where the second inequality is H\"older's inequality. Since $0<\int \varphi^p\,dm<\infty$, we get that $\int\varphi^p\,dm\leq\int\psi^p\,dm.$ Thus, $\varphi$ is extremal for the $p$-modulus of $\mathbf{E}\cup\mathbf{F}$. Finally, $\Mod_p\mathbf{E}\leq\int\varphi^p\,dm=\Mod_p\mathbf{E}\cup\mathbf{F}=\Mod_p\mathbf{E}$, by (a). Therefore, $\varphi$ is extremal for the $p$-modulus of $\mathbf{E}$.

For the reverse direction, we require a short lemma.

\begin{lemma}\label{plemma} Let $1<p<\infty$. If $\varphi, f\in L^p(m)$ take values in $[-\infty,\infty]$ and $\varphi\geq 0$, then $$ \int \left[(\varphi+\varepsilon f)^+\right]^p\,dm= \int_{\{\varphi+\varepsilon f>0\}} \left[\varphi^p+ p\,\varepsilon f\varphi^{p-1}\right]\,dm+o(\varepsilon)\cdot\varepsilon $$ where $o(\varepsilon)\rightarrow 0$ as $\varepsilon\rightarrow 0$.\end{lemma}

\begin{proof} Let $1<p<\infty$ and let $\varphi,f\in L^p(m)$ be given. Assume that the functions $\varphi$ and $f$ take values in $[-\infty,\infty]$ and $\varphi\geq 0$. Fix $\varepsilon\neq 0$ and set $P=\{\varphi+\varepsilon f>0\}$. By the mean value theorem, for all $x\in P$ such that  $\varphi(x)$ and $f(x)$ are both finite, there exists $\delta=\delta(x)$ between $0$ and $\varepsilon$ such that $(\varphi+\varepsilon f)^p -\varphi^{p} =\left[p(\varphi+\delta f)^{p-1}f\right]\varepsilon$. In particular, this holds at $m$-a.e.~ $x\in P$, because $\varphi,f\in L^p(m)$, and the function $\delta:P\rightarrow\RR$ is measurable, because $\varphi$ and $f$ are measurable. Hence $$\int_P (\varphi+\varepsilon f)^p\,dm= \int_P \left[\varphi^p +  p\,\varepsilon f\varphi^{p-1}\right]\,dm+\varepsilon\int_P pf\left[(\varphi+\delta f)^{p-1}-\varphi^{p-1}\right]\,dm.$$ The lemma follows, because the second integral in the displayed equation vanishes as $\varepsilon\rightarrow 0$ by the dominated convergence theorem.
\end{proof}

Now suppose that $\varphi$ is extremal for the $p$-modulus of $\mathbf{E}$. Break $\mathbf{E}=\mathbf{E}_0\cup\mathbf{E}_\infty$ into a union of two measure systems where $\mathbf{E}_0=\{\mu\in\mathbf{E}:1\leq \int\varphi\,d\mu<\infty\}$ and $\mathbf{E}_\infty=\{\mu\in\mathbf{E}:\int\varphi\,d\mu=\infty\}$. Since $\varphi\in L^p(m)$, we have $\Mod_p\mathbf{E}_\infty=0$, because $\varepsilon \varphi$ is admissible for $\mathbf{E}_\infty$ for all $\varepsilon>0$. It follows that $\Mod_p\mathbf{E}_0=\Mod_p\mathbf{E}=\int\varphi^p\,dm$; that is, $\varphi$ is extremal for the $p$-modulus of $\mathbf{E}_0$, as well. Moreover, $\mathbf{E}_0$ is nonempty, since $\Mod_p\mathbf{E}_0>0$. Recall that we want to show that condition $(B_p)$ holds. Assign $\mathbf{F}$ to be the family of all measures $\nu$ defined on $\mathcal{M}$ such that $\int \varphi\,d\nu=1$. Thus, (b) is satisfied by the definition of $\mathbf{F}$. To verify (a), simply note that $$\Mod_p\mathbf{E}\leq \Mod_p \mathbf{E}\cup\mathbf{F}\leq \int \varphi^p\, dm= \Mod_p\mathbf{E},$$ since $\varphi$ is admissible for $\mathbf{E}\cup\mathbf{F}$ and $\varphi$ is extremal for the $p$-modulus of $\mathbf{E}$. It remains to establish (c).  Assume that $f\in L^p(m)$ takes values in $[-\infty,\infty]$  and $\int f\, d\nu\geq 0$ for every $\nu\in\mathbf{F}$. Then for all
$\varepsilon>0$ the metric $\varphi_\varepsilon=(\varphi+\varepsilon f)^+\geq 0$ belongs to $L^p(m)$ and $\int \varphi_\varepsilon\, d\nu\geq \int (\varphi+\varepsilon f)\, d\nu\geq \int \varphi\, d\nu=1$ for every $\nu\in\mathbf{F}$. If $\mu\in\mathbf{E}_0$, then there exists $0<c\leq 1$ such that $c\mu\in\mathbf{F}$ so that $\int \varphi_\varepsilon\,d\mu\geq c\int\varphi_\varepsilon\,d\mu=\int\varphi_{\varepsilon}\,d(c\mu) \geq 1.$ Hence the metric $\varphi_\varepsilon$ is also admissible for $\mathbf{E}_0$. Thus, $$\int \varphi^p\,dm=\Mod_p\mathbf{E}_0\leq \int \varphi^p_{\varepsilon}\,dm=\int \left[(\varphi+\varepsilon f)^+\right]^p\,dm.$$ Then, Lemma \ref{plemma} gives $\int \varphi^p\,dm\leq \int_{P_\varepsilon} \left[\varphi^p+ p\,\varepsilon f\varphi^{p-1}\right]\,dm+o(\varepsilon)\cdot\varepsilon$,  where the set $P_\varepsilon=\{\varphi+\varepsilon f>0\}$ and $o(\varepsilon)\rightarrow 0$ as $\varepsilon\rightarrow 0$. It follows that $$p\,\varepsilon\int_{P_\varepsilon} f\varphi^{p-1}\,dm \geq \int_{X\setminus P_\varepsilon}\varphi^p\,dm-o(\varepsilon)\cdot \varepsilon\geq -o(\varepsilon)\cdot \varepsilon.$$ Dividing through by $p\,\varepsilon$ and letting $\varepsilon\rightarrow 0+$, we obtain $$\int f\varphi^{p-1}\,dm=\lim_{\varepsilon\rightarrow 0+}\int_{P_{\varepsilon}} f\varphi^{p-1}\,dm\geq 0,$$ by the dominated convergence theorem. Therefore, condition $(B_p)$ holds if $\varphi$ is extremal for the $p$-modulus of $\mathbf{E}$.

\section{Proof of Theorem \ref{L1Beurling} (Extremal Metrics in $L^1$)}\label{l1sect}

Let $\mathbf{E}$ be a measure system and let $\varphi$ be an admissible metric for $\mathbf{E}$ such that $\varphi\in L^1(m)$. If $\varphi=0$ $m$-a.e., then $\varphi$ is extremal for the $1$-modulus of $\mathbf{E}$ and condition $(B_1)$ holds with $\mathbf{F}=\emptyset$. Thus, we assume that $0<\int \varphi\,dm<\infty$.

Suppose that condition $(B_1)$ holds. Let $\psi$ be an admissible metric for $\mathbf{E}\cup\mathbf{F}$ with $\psi\in L^1(m)$. Then
$\int \psi\,d\nu \geq 1=\int \varphi\, d\nu$ for every $\nu\in\mathbf{F}$, by (b). Hence the function $f=\psi-\varphi\in L^1(m)$ takes values in $[-\infty,\infty]$, $\int f\,d\nu\geq 0$ for all $\nu\in\mathbf{F}$ and $f(x)\geq 0$ whenever $\varphi(x)=0$. Since $f$ satisfies the hypothesis of (c), we obtain $\int (\psi-\varphi)\,dm \geq 0$. That is, $\int\varphi\,dm\leq \int \psi\, dm$, for every admissible metric $\psi$. Thus, $\varphi$ is extremal for the $1$-modulus of $\mathbf{E}\cup\mathbf{F}$. It follows that $\Mod_1\mathbf{E}\leq \int\varphi\,dm=\Mod_1\mathbf{E}\cup\mathbf{F}=\Mod_1\mathbf{E}$, by (a). Therefore, $\varphi$ is extremal for the $1$-modulus of $\mathbf{E}$.

Conversely, suppose that $\varphi$ is extremal for the 1-modulus of $\mathbf{E}$. Then $\varphi$ is also extremal for the 1-modulus of $\mathbf{E}_0=\{\mu\in\mathbf{E}:1\leq \int\varphi\,d\mu<\infty\}$. We want to check that condition $(B_1)$ holds. Assign $\mathbf{F}$ to be the family of all measures $\nu$ defined on $\mathcal{M}$ such that $\int \varphi\,d\nu=1$. Then (b) is satisfied automatically. For (a), $\Mod_1\mathbf{E}\leq \Mod_1 \mathbf{E}\cup\mathbf{F}\leq \int \varphi\, dm= \Mod_1\mathbf{E},$ since $\varphi$ is admissible for $\mathbf{E}\cup\mathbf{F}$ and $\varphi$ is extremal for the $1$-modulus of $\mathbf{E}$. It remains to verify (c). Assume that $f\in L^1(m)$ takes values in $[-\infty,\infty]$, $\int f\, d\nu\geq 0$ for all $\nu\in\mathbf{F}$ and $f(x)\geq 0$ whenever $\varphi(x)=0$.
For all $\varepsilon>0$ the metric $\varphi_\varepsilon=(\varphi+\varepsilon f)^+\geq 0$ belongs to $L^1(m)$, and moreover, satisfies $\int \varphi_\varepsilon\, d\nu\geq \int (\varphi+\varepsilon f)\, d\nu\geq \int \varphi\, d\nu=1$ for every $\nu\in\mathbf{F}$. Now, for all $\mu\in\mathbf{E}_0$, there exists $0<c\leq 1$ such that $c\mu\in\mathbf{F}$. Thus, $\int \varphi_\varepsilon\,d\mu\geq c\int\varphi_\varepsilon\,d\mu=\int\varphi_{\varepsilon}\,d(c\mu) \geq 1$ for all $\mu\in\mathbf{E}_0$. This shows that the metric $\varphi_\varepsilon$ is also admissible for $\mathbf{E}_0$, and hence, $$\int \varphi\,dm=\Mod_1\mathbf{E}_0\leq \int \varphi_{\varepsilon}\,dm=\int (\varphi+\varepsilon f)^+\,dm=\int_{P_\varepsilon}(\varphi+\varepsilon f)\,dm,$$ where $P_\varepsilon=\{\varphi+\varepsilon f>0\}$.
This yields $\int_{P_\varepsilon} f\,dm \geq \varepsilon^{-1}\int_{X\setminus P_\varepsilon} \varphi\,dm\geq 0$ for all $\varepsilon>0$. As $\varepsilon\rightarrow 0+$, the characteristic functions $\chi_{P_\varepsilon}$ converge $m$-a.e.~ to the function $\chi_P$ where $P=\{\varphi>0\}\cup\{\varphi=0,f>0\}$ (convergence at $x\in X$ fails if $f(x)=-\infty$). Therefore, $\int_P f\,dm\geq 0$, and because we assumed that $f(x)\geq 0$ whenever $\varphi(x)=0$, we obtain $\int f\,dm=\int_P f\,dm+\int_{\{\varphi=0,f=0\}}f\,dm\geq 0$, as well. This completes the proof that condition $(B_1)$ holds whenever $\varphi$ is extremal for the $1$-modulus of $\mathbf{E}$.

\section{Modification for Curve Families in $\RR^n$}\label{modify}

The conditions $(B'_p)$ of Corollaries \ref{Lpcurve} and \ref{L1curve} are sufficient tests for metrics to be extremal for $\Mod_p\Gamma$ by Theorems \ref{LpBeurling} and \ref{L1Beurling}. To verify that the conditions $(B'_p)$ are also necessary, the proofs of Theorems \ref{LpBeurling} and \ref{L1Beurling} can be modified, as follows.

Let $1\leq p<\infty$. Let $\Gamma\subset\RR^n$ be a curve family and let $\varphi$ be an extremal metric for the $p$-modulus of $\Gamma$ such that $0<\int_{\RR^n}\varphi^p<\infty$. Then the metric $\varphi$ is also extremal for the $p$-modulus of $\Gamma_0=\{\gamma\in\Gamma:1\leq\int_\gamma\varphi\,ds<\infty\}$. We want to check that condition $(B_p')$ holds. Assign $\Gamma'$ to be the family of all curves $\gamma$ in $\RR^n$ such that $\int_\gamma\varphi\,ds=1$. Then (b) holds by definition. To show (a), simply note that $\Mod_p\Gamma\leq\Mod_p\Gamma\cup\Gamma' \leq \int_{\RR^n}\varphi^p = \Mod_p\Gamma$, because $\varphi$ is admissible for $\Gamma\cup\Gamma'$ and $\varphi$ is extremal for the $p$-modulus of $\Gamma$. It remains to verify (c). Assume that $f\in L^p(\RR^n)$ takes values in $[-\infty,\infty]$ and $\int_\gamma f\,ds\geq 0$ for all $\gamma\in\Gamma'$. In the case $p=1$, also assume that $f(x)\geq 0$ whenever $\varphi(x)=0$. For all $\varepsilon>0$, the metric $\varphi_\varepsilon=(\varphi+\varepsilon f)^+\geq 0$ belongs to $L^p(\RR^n)$. Moreover, $\int_\gamma \varphi_\varepsilon\, ds\geq \int_\gamma (\varphi+\varepsilon f)\, ds\geq \int_\gamma \varphi\, ds=1$ for all $\gamma\in\Gamma'$. If $\gamma\in\Gamma_0$, then $$1\leq \int_\gamma\varphi\,ds = \sum_{i}\int_{a_i}^{b_i}\varphi(\gamma_i(t))|\gamma_i'(t)|dt<\infty.$$ Since each term in the line integral is non-negative and finite, the function $$c\mapsto \int_{a_i}^c \varphi(\gamma_i(t))|\gamma_i'(t)|dt$$ on $[a_i,b_i]$ is continuous for each $i$. Hence we can pick $c_i\in[a_i,b_i]$ for all $i$ in such a way that the subcurve $\gamma_1=\bigsqcup_i \gamma([a_i,c_i])$ of $\gamma$ satisfies $\int_{\gamma_1}\varphi\,ds=1$. This means that $\gamma_1\in\Gamma'$. Thus, $\int_{\gamma}\varphi_\varepsilon\,ds\geq \int_{\gamma_1}\varphi_\varepsilon\,ds\geq 1$. This shows that $\varphi_\varepsilon$ is also admissible for $\Gamma_0$. Hence $$\int_{\RR^n} \varphi^p=\Mod_p\Gamma_0\leq \int_{\RR^n} \varphi_{\varepsilon}^p=\int_{\RR^n} \left[(\varphi+\varepsilon f)^+\right]^p.$$ To finish checking (c), one can now proceed as above. Follow the argument from section \ref{lpsect}, when $1<p<\infty$, and follow the argument from section \ref{l1sect}, when $p=1$.

\section{Proof of Theorem \ref{curvemetrics}}\label{sect4}

Suppose that $\varphi:\RR^n\rightarrow[0,\infty]$ is a Borel function and let $\Gamma_\varphi$ be the family of all curves $\gamma$ in $\RR^n$ such that $\int_{\gamma}\varphi\,ds\geq 1$. Fix any $1\leq p<\infty$ such that $0<\int_{\RR^n}\varphi^p<\infty$. We want to show that $\varphi$ is extremal for the $p$-modulus of $\Gamma_\varphi$. For each $y\in\RR^n$ let $\ell_y=y+\RR e_1\cong\RR$ denote the line through $y$ parallel to the direction $e_1=(1,0\dots,0)$. By Fubini's theorem, $\varphi\in L^p(\ell_y)$, $y=(0,\bar y)$ for $H^{n-1}$-a.e.~ $\bar y\in\RR^{n-1}$. In particular, we also have $\varphi\in L^1_{\rm loc}(\ell_y)$, $y=(0,\bar y)$ for $H^{n-1}$-a.e.~ $\bar y\in\RR^{n-1}$. Here, as above and as below, $H^s$ denotes $s$-dimensional Hausdorff measure. Below $|I|$ denotes the diameter of an interval $I$.

\begin{lemma}\label{ilemma} Suppose that $\varphi\in L^1_{\rm loc}(\ell_y)$. Then, for $H^1$-a.e.~ $x\in \ell_y$ such that $\varphi(x)>0$, there exist a sequence of positive integers $n_k=n_k(x)\rightarrow\infty$ and a sequence intervals $I_k=I_k(x)\subset \ell_y$ centered at $x$ with $|I_k|\rightarrow 0$ such that $\int_{I_k}\varphi\, dt=1/n_k$ for all $k$.
\end{lemma}

\begin{proof} Define the function $g_x(r)=\int_{-r}^r\varphi(x+te_1)\,dt$ for all $x\in\ell_y$ and $r\geq 0$. Then $\lim_{r\rightarrow 0+} g_x(r)/2r= \varphi(x)$ for $H^1$-a.e.~ $x\in\ell_y$, by the Lebesgue differentiation theorem.
Hence for $H^1$-a.e.~ $x\in\ell_y$ such that $\varphi(x)>0$, there exists $r_0=r_0(x)>0$ such that $0<g_x(r)<\infty$ for all $0<r\leq r_0$. Since $g_x|_{[0,r_0]}$ is continuous and $g_x(0)=0$, we can find a sequence of integers $n_k=n_k(x)$ and a sequence of radii $r_k=r_k(x)\rightarrow 0$ such that $g_x(r_k)=1/n_k$. Then $I_k=I_k(x)=x+[-r_k,r_k]e_1\subset\ell_y$ is a sequence of intervals with the desired property.\end{proof}

Let $E\subset\RR^n$ be the set of points $x\in\RR^n$ where the conclusion of Lemma \ref{ilemma} holds, i.e.~ $x\in E$ if and only if there exists a sequence of positive integers $n_k=n_k(x)\rightarrow\infty$ and a sequence of intervals $I_k=I_k(x)\subset\ell_x$ centered at $x$ with $|I_k|\rightarrow 0$ such that $\int_{I_k}\varphi\,dt=1/n_k$. By Fubini's theorem and Lemma \ref{ilemma}, we have $x\in E$ for a.e.~ $x\in\RR^n$ such that $\varphi(x)>0$. We define a curve family $\Gamma'\subset\Gamma_\varphi$ as follows. Choose one pair of sequences $(n_k(x))_{k=1}^\infty$ and $(I_k(x))_{k=1}^\infty$ for each $x\in E$. Then, for each $x\in E$ and $k\geq 1$, define a curve $\gamma_k(x)=\bigsqcup_{i=1}^{n_k} I_k(x)$, i.e.~ let $\gamma_k(x)$ be a curve which covers the interval $I_k(x)$ exactly $n_k(x)$-times. Set $\Gamma'=\{\gamma_k(x):x\in E$ and $k\geq 1\}$.

To prove that $\varphi$ is extremal for the $p$-modulus of $\Gamma_\varphi$, it is enough by either Corollary \ref{Lpcurve} or Corollary \ref{L1curve} (according to whether $1<p<\infty$ or $p=1$) to show that $(B_p')$ holds for $\Gamma'$. To start note $\int_{\gamma_k(x)}\varphi\,ds=n_k(x)\int_{I_k(x)}\varphi\,dt=n_k(x)/n_k(x)=1$ for all $\gamma_k(x)\in\Gamma'$. This shows that (b) holds. And, since $\Gamma'\subset\Gamma_\varphi$, (a) is true too. To verify (c), assume that $f\in L^p(\RR^n)$ takes values in $[-\infty,\infty]$ and $\int_{\gamma_k(x)} f\,ds\geq 0$ for all $\gamma_k(x)\in\Gamma'$. Moreover, if $p=1$, assume that $f(x)\geq 0$ whenever $\varphi(x)=0$. By Fubini's theorem and the Lebesgue differentiation theorem, $$f(x)=\lim_{k\rightarrow\infty} \frac{1}{|I_k(x)|}\int_{I_k(x)} f\,dt$$ for a.e.~ $x\in E$, and in particular, for a.e.~ $x\in\RR^n$ such that $\varphi(x)>0$. By assumption, $$\int_{I_k(x)} f\,dt=\frac{1}{n_k(x)}\int_{\gamma_k(x)} f\,ds \geq 0\quad\text{for all }x\in E\text{ and }k\geq 1.$$ Thus, combining the two displayed equations, $f(x)\geq 0$ at a.e.~ $x\in\RR^n$ such that $\varphi(x)>0$. It follows that $\int_{\RR^n} f\varphi^{p-1} \geq 0$, if $1<p<\infty$, and $\int_{\RR^n} f\geq 0$, if $p=1$. Hence (c) holds. Therefore, $(B_p')$ holds and $\varphi$ is extremal for the $p$-modulus of $\Gamma_\varphi$.

\end{document}